\documentclass[12pt]{article}
\input epsf.tex

\usepackage{amsmath}
\usepackage{amsthm}
\usepackage{amsfonts}
\usepackage{amssymb}
\usepackage{graphicx}
\usepackage{latexsym}

\usepackage{amsmath,amsthm,amsfonts,amssymb}

\usepackage{epsfig}

\usepackage{amssymb}
\usepackage[all]{xy}
\xyoption{poly}
\usepackage{fancyhdr}
\usepackage{wrapfig}

\theoremstyle{plain}
\newtheorem{thm}{Theorem}[section]

\newtheorem{lem}[thm]{Lemma}
\newtheorem{cor}[thm]{Corollary}

\theoremstyle{definition}

\theoremstyle{remark}

\advance \topmargin by -\headheight
\advance \topmargin by -\headsep
\evensidemargin \oddsidemargin
\def\Aut{\textrm{Aut}}

\def\talpha{{\tilde \alpha}}

\def\talpha{{\tilde \alpha}}

\def\tbeta{\tilde \beta}

\def\cB{{\mathcal B}}

\def\cC{{\mathcal C}}

\def\cc{{\curvearrowright}}

\def\dom{\textrm{Dom}}

\def\cR{\mathcal R}

\def\chix{{\raise.5ex\hbox{$\chi$}}}

\def\Z{{\mathbb Z}}

\begin{document}
\title{Every countably infinite group is almost Ornstein }
\author{Lewis Bowen\footnote{email:lpbowen@math.tamu.edu} \\ Texas A\&M University\\ {\small {\em Dedicated to Anatoly Stepin on the occasion of his 70th birthday}}}
\begin{abstract}
We say that a countable discrete group $G$ is {\em almost Ornstein} if for every pair of standard non-two-atom probability spaces $(K,\kappa), (L,\lambda)$ with the same Shannon entropy,  the Bernoulli shifts $G \cc (K^G,\kappa^G)$ and $G \cc (L^G,\lambda^G)$ are isomorphic. 

This paper proves every countably infinite group is almost Ornstein.
\end{abstract}
\maketitle

\noindent

\section{Introduction}

All probability spaces in this paper are standard. Let $G$ be a countable discrete group and $(K,\kappa)$ a probability space. Let $K^G$ be the space of maps $x:G\to K$ with the product measure $\kappa^G$. The group acts on this space by $gx(f)=x(g^{-1}f)$ for every $f,g\in G$ and $x\in K^G$. The dynamical system $G \cc (K^G,\kappa^G)$ is called the {\em Bernoulli shift} over $G$ with {\em base space} $(K,\kappa)$. If $(L,\lambda)$ is another probability space then a measurable map $\phi: K^G \to L^G$ is {\em shift-equivariant} if $\phi(gx)=g\phi(x)$ for every $g\in G$ and a.e. $x\in K^G$. It is an {\em isomorphism} if in addition it has a measurable inverse and $\phi_* \kappa^G =\lambda^G$. 

In the early years of measurable dynamics, von Neumann asked whether $\Z \cc (\{0,1\}^\Z, u_2^\Z)$ is isomorphic to $\Z \cc (\{0,1,2\}^\Z, u_3^\Z)$ where $u_i$ is the uniform probability measure on $\{0,1,\ldots, i-1\}$. Kolmogorov \cite{Ko58, Ko59} answered this question by introducing dynamical entropy for general probability measure-preserving transformations and computing this invariant for Bernoulli shifts over $\Z$. To be precise, the {\em Shannon entropy} of a probability space $(K,\kappa)$ is defined as follows. If there exists a countable set $K' \subset K$ such that $\kappa(K')=1$ then
$$H(K,\kappa) = -\sum_{k\in K'} \kappa(\{k\}) \log( \kappa(\{k\}))$$
where, by convention, $0\log(0)=0$. Otherwise, $H(K,\kappa)=+\infty$. Kolmogorov proved that {\em if} $\Z \cc (K^\Z,\kappa^\Z)$ is isomorphic to $\Z \cc (L^\Z,\lambda^\Z)$ then $H(K,\kappa)=H(L,\lambda)$, thereby answering von Neumann's question in the negative.

In \cite{Or70a, Or70b}, Ornstein famously proved the converse:
\begin{thm}[Ornstein]
Let $(K,\kappa), (L,\lambda)$ be two probability spaces with $H(K,\kappa)=H(L,\lambda)$. Then $\Z \cc (K^\Z,\kappa^\Z)$ is isomorphic to the $\Z \cc (L^\Z,\lambda^\Z)$.
\end{thm}

Motivated by this result, Stepin \cite{St76} made the following definition: a countable group $G$ is {\em Ornstein} if for every pair of  probability spaces $(K,\kappa), (L,\lambda)$ with the same Shannon entropy, the Bernoulli shift $G \cc (K^G,\kappa^G)$ is isomorphic to $G\cc (L^G,\lambda^G)$. Note that no finite group is Ornstein. However, Ornstein and Weiss \cite{OW87} proved that every countably infinite amenable group is Ornstein. Using a co-induction argument, Stepin \cite{St75} observed that if $G$ has an Ornstein subgroup $H$ then $G$ is itself Ornstein. Therefore, every group which contains an infinite amenable subgroup (for example, an infinite cyclic subgroup) is Ornstein. However, there are countable groups which do not satisfy this property (for example, Ol'shanskii's monster \cite{Ol91}). 

A probability space $(K,\kappa)$ is a {\em two-atom space} if $\kappa$ is supported on two atoms: i.e., there exist elements $k_0,k_1 \in K$ such that $\kappa(\{k_0,k_1\})=1$. We say that a group $G$ is {\em almost Ornstein} if whenever $(K,\kappa), (L,\lambda)$ are standard probability spaces, neither of which is a two-atom space and $H(K,\kappa)=H(L,\lambda)$ then $G \cc (K^G,\kappa^G)$ is isomorphic to $G \cc (L^G,\lambda^G)$. So $G$ is almost Ornstein if it is Ornstein with the possible exception of the 1-parameter family of two-atom probability spaces. Our main results are:

\begin{thm}\label{thm:main}
Every countably infinite group is almost Ornstein.
\end{thm}

\begin{thm}\label{thm:2atom}
If $G$ is a countable group, $C < G$ is a cyclic subgroup of prime order and the normalizer of $C$ in $G$ has infinite index, then $G$ is Ornstein.
\end{thm}

Unfortunately, Theorem \ref{thm:2atom} is not sufficient to conclude that every infinite countable group is Ornstein. There is a group $G$, constructed in
Theorem 31.8 of \cite{Ol91} which is generated by 2 elements and splits as a
central extension
$$1 \to C \to G \to T \to 1$$
where $C = \Z/p\Z$ ($p$ is a prime $> 10^{70}$) and $G$ and $T$ are Tarski
 Monsters (meaning that all proper subgroups of $G$ and $T$ are finite and cyclic). In addition, every nontrivial subgroup of $G$ contains $C$. Therefore, $G$ does not contain any cyclic subgroup of prime order with infinite-index normalizer. According to D. Osin \cite{Os11} it is possible to modify the construction (in the spirit of section 37 of \cite{Ol91}) to ensure that $G$ is also non-amenable. Therefore, it does not contain any infinite amenable subgroups and we cannot say whether or not it is Ornstein. However we can say:
 
 \begin{cor}
If every countable infinite group $G$ with a nontrivial center is Ornstein, then every countably infinite group is Ornstein.
\end{cor}
\begin{proof}
Assuming the hypothesis, let $G$ be an arbitrary countably infinite group. We need to show it is Ornstein. By Stepin's Theorem \cite{St75}, we may assume $G$ does not contain any infinite cyclic subgroups. Therefore, it contains a cyclic subgroup $C$ of prime order. By Theorem \ref{thm:2atom}, we may assume the normalizer of $C$ has finite index in $G$. Because $C$ is finite, its centralizer has finite index in its normalizer. Therefore, the centralizer $H$ of $C$ has finite index in $G$ and so $H$ is infinite. By hypothesis, $H$ is Ornstein. By Stepin's Theorem \cite{St75}, this implies $G$ is Ornstein.
\end{proof}

In \cite{Bo10} (see also \cite{KL1, KL2} for an alternative approach) it is shown that if $G$ is any sofic group, then the entropy of the base space is an invariant. Therefore:
\begin{cor}
If $G$ is a countably infinite sofic group and $(K,\kappa), (L,\lambda)$ are two probability spaces, neither of which is a two-atom space, then the Bernoulli shifts $G \cc (K^G,\kappa^G)$ and $G\cc (L^G,\lambda^G)$ are isomorphic if and only if $H(K,\kappa)=H(L,\lambda)$.
\end{cor}

Recall that a shift-equivariant map $\phi: K^G \to L^G$ is a {\em factor map} between Bernoulli shifts $G \cc (K^G,\kappa^G)$ and $G \cc (L^G,\lambda^G)$ if $\phi_*\kappa^G = \lambda^G$. In this case, we say that $G \cc (K^G,\kappa^G)$ {\em factors} onto $G \cc (L^G,\lambda^G)$. Sinai \cite{Si59, Si62} showed that the Bernoulli shift $\Z \cc (K^\Z, \kappa^\Z)$ factors onto $\Z \cc (L^\Z,\lambda^\Z)$ if and only if $H(K,\kappa) \ge H(L,\lambda)$. This result extends to countably infinite amenable groups by \cite{OW87}. Non-amenable groups behave in a very different manner:

\begin{cor}
If $G$ is a countable non-amenable group then there exists a number $0\le r(G) <\infty $ such that if $(K,\kappa)$ is any non-two-atom probability space and $H(K,\kappa)>r(G)$ then $G \cc (K^G,\kappa^G)$ factors onto every Bernoulli shift over $G$. 
\end{cor}
In \cite{Bo11}, it is shown that  if $G$ contains a non-abelian free subgroup then every Bernoulli shift over $G$ factors onto every Bernoulli shift over $G$. It is an open question whether every non-amenable group satisfies this property. 
\begin{proof}
Let $r(G)$ be the infimum over all numbers $r$ such that there exists a  probability space $(K,\kappa)$ with $H(K,\kappa)=r$ such that $G\cc (K^G,\kappa^G)$ factors onto every Bernoulli shift over $G$. It follows from \cite{Ba05} Theorem 6.4, that $r(G)<\infty$ if $G$ is finitely generated. If $G$ is not finitely generated then it contains a finitely generated non-amenable subgroup $H$. An easy co-induction argument shows that $r(G) \le r(H)<\infty$ (see e.g., \cite{Bo11} which shows how $G \cc (K^G,\kappa^G)$ is co-induced from $H \cc (K^H,\kappa^H)$). 

We claim that if $(K,\kappa)$ is a non-two-atom probability space  and $(L,\lambda)$ is an arbitrary probability space with $H(K,\kappa) > H(L,\lambda)$ then $G\cc (K^G, \kappa^G)$ factors onto $G\cc (L^G,\lambda^G)$. To see this, let $(M,\mu)$ be a probability space with $H(K,\kappa) = H(M,\mu)+ H(L,\lambda)$. By Theorem \ref{thm:main}, $G\cc (K^G, \kappa^G)$ is isomorphic to $G \cc ((M\times L)^G, (\mu \times \lambda)^G)$ which clearly factors onto $G\cc (L^G,\lambda^G)$. 

Now let $(M,\mu)$ be a non-two-atom probability space with $H(M,\mu) > r(G)$. Then there exists a probability space $(K,\kappa)$ with $H(K,\kappa) < H(M,\mu)$ such that $G \cc (K^G,\kappa^G)$ factors onto every Bernoulli shift. Since $G \cc (M^G,\mu^G)$ factors onto $G \cc (K^G,\kappa^G)$, it follows that $G\cc (M^G,\mu^G)$ factors onto every Bernoulli shift as claimed.


\end{proof}

The main ingredients of the proof of Theorems \ref{thm:main} and \ref{thm:2atom} are (i) Thouvenot's relative isomorphism theorem for actions of $\Z$ \cite{Th75}, (ii) the fact that the full group of any p.m.p. aperiodic equivalence relation contains an aperiodic automorphism, (iii) a co-induction argument similar in spirit to Stepin's \cite{St75}. The idea to use elements in the full group of a factor to obtain isomorphism theorems originated in \cite{RW00}  and was applied in \cite{DP02} to obtain a version of Thouvenot's relative isomorphism theorem for actions of amenable groups.


{\bf Acknowledgements}. I'd like to thank Brandon Seward for sketching a proof of Theorem \ref{thm:T}. Thanks also to Hanfeng Li for pointing out several errors in a previous version. Also thanks to Benjy Weiss for helpful suggestions which improved the paper and corrected an error. I am indebted to Denis Osin for explaining that Ol'shankii's monsters are still potential counterexamples to the claim that every countably infinite group is Ornstein.

\section{Preliminaries}
All probability spaces in this paper are standard and may be atomic or non-atomic.  Often we denote a probability space by $(X,\mu)$ without referencing the sigma-algebra. All maps, functions, relations, etc., are considered up to sets of measure zero. 

\subsection{Entropy} 
Let $(X,\cB,\mu)$ be a standard probability space. Let $\chi:X \to K$ be a measurable map. The entropy of $\chi$ is $H(\chi) = H(K, \chi_*\mu)$ (i.e., it is the entropy of the partition $\chi^{-1}(P)$ where $P$ is the partition of $K$ into points). If $\cC \subset \cB$ is a sub-sigma algebra, let $H(\chi| \cC)$ denote the relative entropy. 

Let $T\in \Aut(X,\cB_X,\mu)$ be an automorphism of $(X,\cB_X,\mu)$. Let $\cC \subset \cB_X$ be a $T$-invariant sub-sigma-algebra. Let $h(T, \chi | \cC) = \lim_{n\to\infty} (2n+1)^{-1} H(\bigvee_{i = -n}^n \chi \circ T^i |\cC)$ denote the relative entropy rate of $\chi$. Let $h(T|\cC) = \sup h(T,\chi|\cC)$ where the supremum is over all measurable maps $\chi$ with finite range.


A measurable map $\chi:X \to K$ is a {\em generator} for $(T,X,\cB_X,\mu)$ if $\cB_X$ is the smallest $T$-invariant sigma-algebra under which $\chi$ is measurable. In this case, if $\cC$ is any $T$-invariant sub-sigma-algebra then the Kolmogorov-Sinai Theorem implies that if $H(\chi |\cC)<\infty$ then $h(T|\cC)=h(T,\chi|\cC)$.


\subsection{Thouvenot's relative isomorphism theorem}

Let $T \in \Aut(X,\cB_X,\mu)$, $U \in \Aut(Z,\cB_Z,\zeta)$ and suppose $\pi:X \to Z$ is a factor map. That is, $\pi$ is measurable, $\pi_*\mu=\zeta$ and $\pi T = U \pi$. Then we say that $(T,X,\cB_X,\mu)$ is {\em Bernoulli} relative to $(U,Z,\cB_Z,\zeta)$ if there is a generator $\chi:X \to K$ for $(T,X,\cB_X,\mu)$ such that the random variables $\{\chi \circ T^i:~i \in \Z\}$ are jointly independent relative to $\pi^{-1}(\cB_Z)$. The dependence on $\pi$ in this definition is left implicit. The next result is in \cite{Th75}.

\begin{thm}[Thouvenot]\label{thm:Thouvenot}
Let $T \in \Aut(X,\cB_X,\mu)$, $S \in \Aut(Y,\cB_Y, \nu)$, $U \in \Aut(Z,\cB_Z,\zeta)$ and $\pi_X:X \to Z$, $\pi_Y:Y \to Z$ be factor maps. Suppose $(T,X,\cB_X,\mu)$ and $(S,Y,\cB_Y,\nu)$ are each Bernoulli relative to $(U,Z,\cB_Z,\zeta)$ and $U$ is ergodic. Suppose also that $h(T| \pi_X^{-1}(\cB_Z)) = h(S| \pi_Y^{-1}(\cB_Z))$. Then there is a measure-space isomorphism $\phi:( X, \cB_X,\mu) \to (Y,\cB_Y,\nu)$ such that $\phi T = S \phi$ and $\pi_X = \pi_Y \phi$.  
\end{thm}

\subsection{Measured equivalence relations}

A {\em measurable equivalence relation} on a Borel space $X$ is an equivalence relation $E$ on $X$ such that  $E$ is  a Borel subset of  $X \times X$. If $(x,y)\in E$ we write $xEy$. We say that $E$ is {\em countable} if every $E$-equivalence class is countable. An {\em inner automorphism} of $E$ is a Borel isomorphism $\phi:X \to X$ such that the graph of $\phi$ is contained in $E$. The group of all inner automorphisms is called the {\em full group} and denoted by $[E]$. 

Now assume $\mu$ is a probability measure on $X$ so that $(X,\mu)$ is a standard probability space. If $\phi_*\mu=\mu$ for every $\phi \in [E]$ then we say $(X,\mu,E)$ is a {\em probability measure preserving} (p.m.p.) equivalence relation. We also say that $\mu$ is $E$-invariant. For example, if $G \cc (X,\mu)$ is a probability measure-preserving action of a countable group $G$ and $E:=\{(x,gx):~x\in X, g\in G\}$ then $\mu$ is $E$-invariant. We say $(X,\mu,E)$ is {\em aperiodic} if for a.e.  $x\in X$ the $E$-equivalence class of $x$ is infinite. We say $(X,\mu,E)$ is {\em ergodic} if for every measurable set $A \subset X$ which is a union of $E$-classes, either $\mu(A)=0$ or $\mu(A)=1$.

\begin{thm}\label{thm:T}
Let $(X,\mu,E)$ be an aperiodic ergodic p.m.p. equivalence relation. Then there exists an ergodic automorphism $T \in [E]$ such that for a.e. $x\in X$, $\{T^i x:~ i\in \Z\}$ is infinite.
\end{thm}
\begin{proof}
This is a classical ``folk'' theorem. It is a special case of Corollary 3.3 of \cite{Me93}. It is also proven in \cite{Ke10}, Theorem 3.5.

\end{proof}

\section{Almost Ornstein groups}

\begin{lem}\label{lem:key}
Let $G$ be a countably infinite group and let $(K,\kappa), (L,\lambda), (M,\mu)$ be standard probability spaces with $H(K,\kappa)=H(L,\lambda)$.  In addition, suppose there exist measurable maps $\alpha:K \to M$ and $\beta:L \to M$ such that $\alpha_*\kappa=\beta_*\lambda=\mu$ and $(M,\mu)$ is nontrivial (i.e., $H(M,\mu)>0$). Then $G \cc (K^G,\kappa^G)$ is isomorphic to $G \cc (L^G,\lambda^G)$.
\end{lem}

\begin{proof}
 Let $\alpha^G: K^G \to M^G$  be the product map: $\alpha^G(x)(g):=\alpha(x(g))$. Define $\beta^G:L^G \to M^G$ similarly. Note $\alpha^G_*\kappa^G= \beta^G_*\lambda^G=\mu^G$. 


By Theorem \ref{thm:T} there exists an ergodic $U \in \Aut(M^G,\mu^G)$ such that (i) for a.e. $x\in M^G$, the orbit $\{U^ix:~ i \in \Z\}$ is infinite, and (ii)  for a.e. $x\in M^G$ there is a $g\in G$ such that $Ux=gx$.


Define $T:K^G \to K^G$ by $T(x)=gx$ where $g\in G$ is such that $U(\alpha^G(x))=g\alpha^G(x)$. Similarly, let $S:L^G \to L^G$ be defined by $S(y)=gy$ where $g\in G$ is such that $U(\beta^G(y))=g\beta^G(y)$. Note that for a.e. $(x,y)\in K^G \times L^G$ with $\alpha^G(x)=\beta^G(y)$, there is a $g\in G$ such that $Tx = gx$ and $Sy=gy$.

Let $\chi_M: M^G \to M$ be the projection map $\chi_M(x)=x(e)$. Let $\cB_M$ be the smallest $U$-invariant sigma-algebra on $M^G$ for which $\chi_M$ is measurable. Define $\chi_K:K^G \to K, \chi_L: L^G \to L, \cB_K, \cB_L$ similarly. 

Observe that $(T,K^G, \cB_K, \kappa^G)$ and $(S,L^G,\cB_L, \lambda^G)$ are Bernoulli relative to $(U,M^G, \cB_M, \mu^G)$. To see this, note, for example, that $\chi_K$ is a generator for $(T,K^G, \cB_K, \kappa^G)$ and $\{\chi_K \circ T^i : i \in \Z\}$ are jointly independent relative to $(\alpha^G)^{-1}(\cB_M)$. So 
$$h(T| (\alpha^G)^{-1}(\cB_M) ) = H(\chi_K |~ (\alpha^G)^{-1}(\cB_M) ) = H(K,\kappa) - H(M,\mu).$$
Similarly, $H(L,\lambda) - H(M,\mu) = h(S| (\beta^G)^{-1}(\cB_M) )$. By Theorem \ref{thm:Thouvenot}, there is an isomorphism $\phi: (K^G, \cB_K, \kappa^G) \to (L^G,\cB_L, \lambda^G)$ such that $\phi T = S \phi$ and $\alpha^G = \beta^G \phi$. 




Define $\Phi:K^G \to L^G$ by $\Phi(x)(g)= \phi(g^{-1}x)(e)$ for $g\in G$, $x\in K^G$. Observe that $\Phi$ is measurable with respect to the Borel sigma-algebras of $K^G$ and $L^G$ (which are, in general, larger than $\cB_K$ and $\cB_L$). We claim this is the required isomorphism. It is clearly shift-equivariant.

To see that $\Phi$ is invertible, define $\Psi: L^G \to K^G$ by $\Psi(y)(g) = \phi^{-1}(g^{-1}y)(e)$. Because both $\Phi$ and $\Psi$ are shift-equivariant and $\Phi \Psi(y)(e)=y(e)$, $\Psi\Phi(x)(e)=x(e)$, it follows that $\Psi$ is the inverse of $\Phi$. It is easy to check that $\beta^G \Phi = \alpha^G$ and $\Phi T = S \Phi$.



To finish, we must show that $\Phi_*\kappa^G = \lambda^G$. For $z \in M^G$, let $X_z=\{x\in K^G:~ \alpha^G(x)=z\}$ and $Y_z=\{y\in L^G:~\beta^G(y)=z\}$. Because $\beta^G \Phi = \alpha^G$, it follows that $\Phi$ maps $X_z$ to $Y_z$. 




For $z\in M^G$, let $\zeta_z$ be the fiber measure of $\kappa^G$ over $z$. This family of measures is determined (up to measure zero sets) by the property that $\zeta_z$ is supported on $X_z$ and $\kappa^G = \int \zeta_z~d\mu^G(z)$. Similarly, let $\nu_z$ be the fiber measure of $\lambda^G$ over $z$.

Because $\alpha^G_*\kappa^G = \beta^G_*\lambda^G = \mu^G$ and $\Phi$ maps $X_z$ to $Y_z$, in order to show that $\Phi_*\kappa^G = \lambda^G$, it suffices to show that $\Phi$ restricts to an isomorphism from $(X_z,\zeta_z)$ to $(Y_z,\nu_z)$ for a.e. $z$. 



For $g\in G$ and $n\in \Z$ let $\tau_z^n(g) \in G$ be the element satisfying $U^n(g^{-1}z)(e)=z(\tau_z^n(g))$. This is well-defined for a.e. $z\in M^G$. If $x\in X_z$ then $T^n(g^{-1}x)(e)=x(\tau_z^n(g))$ by definition.

Fix $z\in M^G$. Let $e=g_0,g_1,\ldots \in G$ be such that if $O_i = \{\tau^n_z(g_i):~ n \in \Z\}$ then $O_i \cap O_j =\emptyset $ whenever $i\ne j$ and $\cup_{i=0}^\infty O_i=G$. So $K^G = \prod_{i=0}^\infty K^{O_i}$.

Let $\cB_{K,z,0}$ be the restriction of $\cB_K$ to $X_z$. This is the $\sigma$-algebra of $X_z$ generated by $\{\chi_K \circ T^j :~ j \in \Z\}$. More generally, let $\cB_{K,z,i}$ be the sigma-algebra on $X_z$ generated by $\{\chi_K \circ T^j g_i^{-1}:~ j \in \Z\}$. Define $\sigma$-algebras $\cB_{L,z,i}$ on $Y_z$ similarly. 

Because the $O_i$'s are pairwise disjoint, the sigma-algebras $\cB_{K,z,i}$ are independent. Because $\cup O_i = G$, these sigma-algebras generate the Borel sigma-algebra of $K^G$ restricted to $X_z$. Similarly statements apply to $L$ in place of $K$. Therefore, it suffices to show that $\Phi$ restricted to $X_z$ determines an isomorphism from $(X_z,\cB_{K,z,i}, \zeta_z)$ to $(Y_z, \cB_{L,z,i}, \nu_z)$ for all $i$. 

By definition, $\phi$ (and therefore $\Phi$) restricted to $X_z$ is an isomorphism from $(X_z, \cB_{K,z,0}, \zeta_z)$ to $(Y_z, \cB_{L,z,0}, \nu_z)$. Note that $(X_{g_i^{-1}z}, \cB_{K,g^{-1}_iz,0},\zeta_{g_i^{-1}z}) = g^{-1}_i (X_z, \cB_{K,z,i}, \zeta_z)$. Because $\Phi$ is shift-equivariant, this implies  $\Phi$ restricted to $X_z$ determines an isomorphism from $(X_z,\cB_{K,z,i}, \zeta_z)$ to $(Y_z, \cB_{L,z,i}, \nu_z)$ for every $i$.

\end{proof}




\begin{proof}[Proof of Theorem \ref{thm:main}]
Let $G$ be a countably infinite group. Let $(K,\kappa), (L,\lambda)$ be standard probability spaces with $H(K,\kappa)=H(L,\lambda)$. In addition, suppose both $(K,\kappa)$ and $(L,\lambda)$ are not two-atom spaces. We must show that the Bernoulli shifts $G \cc (K^G,\kappa^G)$ and $G \cc  (L^G,\lambda^G)$ are isomorphic.

For $p\in [0,1]$ let $m_p$ be the probability measure on $\{0,1\}$ given $m_p(\{0\})=p$, $m_p(\{1\})=1-p$. If, say $(K,\kappa)$ is not purely atomic then there is some $p_0>0$ such that for all $p<p_0$, $(K,\kappa)$ maps onto $(\{0,1\},m_p)$. In this case, $H(K,\kappa)=\infty=H(L,\lambda)$ which implies that for some $0<p<p_0$, $(L,\lambda)$ also maps onto $(\{0,1\},m_p)$. So the previous lemma implies the result.

Let us now assume that $(K,\kappa)$ and $(L,\lambda)$ are purely atomic. Borrowing an idea from \cite{KS79} (Lemma 2), we observe that if $t>0$ is the largest number such that $\kappa(\{k\})=t$ for some $k\in K$ then there is a number $s>0$ such that $\lambda(\{l\})=s$ for some $l\in L$ and $t+s <1$. Then there is a countable (or finite) set $N$ with a probability measure $\nu$ so that for some $n_0,n_1 \in N$, $\nu(\{n_0\})=t$, $\nu(\{n_1\})=s$ and $H(N,\nu) = H(K,\kappa)=H(L,\lambda)$. In particular, both $(K,\kappa)$ and $(N,\nu)$ map onto $(\{0,1\}, m_t)$. Also $(L,\lambda)$ and $(N,\nu)$ map onto $(\{0,1\}, m_s)$. So the previous lemma implies $G \cc (K^G,\kappa^G)$ is isomorphic to $G \cc (N^G, \nu^G)$ which is isomorphic to $G \cc (L^G, \lambda^G)$. 



\end{proof}











\section{Measurable subgroups}
In order to prove Theorem \ref{thm:2atom} we extend the results of the previous section to so-called {\em measurable subgroups} of a group $G$. To be precise, let $G$ be a countably infinite group and $2^G$ be the set of all subsets of $G$. $G$ acts on $2^G$ by $(g,F)\mapsto gF=\{gf:~f\in F\}$ for $g\in G$ and $F\in 2^G$. Let $\cR$ be the orbit-equivalence relation on $2^G$: $F\cR H \Leftrightarrow \exists g\in G$ such that $F=gH$. Let $2_e^G $ be the set of all $F \in 2^G$ such that $e\in F$. Let $\cR_e = \cR \cap 2_e^G \times 2_e^G$ be the restriction of $\cR$ to $2_e^G$. A {\em measurable subgroup} of $G$ is an $\cR_e$-invariant probability measure $\eta$ on $2_e^G$. 

To justify the definition, note that a subgroup is any subset $H \subset G$ which contains the identity and satisfies $h^{-1}H = H$ for all $h\in H$. A measurable subgroup $\eta$ is the law of a random subset $H$ which contains the identity and has the property that if $H \mapsto h_H \in H$ is a Borel assignment then $h_H^{-1}H$ has the same law as $H$ (as long as $H \mapsto h_H^{-1}H$ is Borel-invertible). For example, if $H$ is a subgroup then $\delta_H$, the Dirac probability measure concentrated at $\{H\}$, is a measurable subgroup.

Let $(K,\kappa)$ be a standard probability space. Let $2^G \otimes K$ be the set of all maps $x:\dom(x) \to K$ with $\dom(x) \subset G$. $G$ acts on this space by $gx(f)=x(g^{-1}f)$ (for $x\in 2^G \otimes K$, $g\in G, f\in g\dom(x)$). Note that $\dom(gx)=g\dom(x)$. Let $\cR_K$ be the orbit-equivalence relation on $2^G \otimes K$: so $x\cR_Ky \Leftrightarrow \exists g\in G$ such that $gx=y$. 

Let $2_e^G \otimes K$ be the set of all $x\in 2^G \otimes K$ with $e\in \dom(x)$. Let $\cR_{e,K}$ be $\cR_K$ restricted to $2_e^G\otimes K$. 

If $\eta$ is a measurable subgroup of $G$ then let $\eta \otimes \kappa$ be the probability measure on $2_e^G \otimes K$ defined by
$$\eta \otimes \kappa = \int \delta_H \times \kappa^H~d\eta(H)$$
where $\delta_H$ is the Dirac measure concentrated on $\{H\}$ and $\kappa^H$ is the product measure on $K^H$. This measure is $\cR_{e,K}$-invariant and projects to $\eta$. It is the {\em Bernoulli shift} over $\eta$ with base space $(K,\kappa)$.

Let $(L,\lambda)$ be another standard probability space. We say the two Bernoulli shifts $(2_e^G \otimes K, \eta \otimes \kappa)$ and $(2_e^G\otimes L, \eta \otimes \lambda)$ are isomorphic if there is a measurable map $\phi: 2_e^G \otimes K \to 2_e^G \otimes L$ such that
\begin{enumerate}
\item $\dom(\phi(x))=\dom(x)$ for a.e. $x$,
\item $\phi_*\eta \otimes \kappa = \eta \otimes \lambda$,
\item $\phi$ is invertible with measurable inverse,
\item $\phi(gx)=g\phi(x)$ for a.e. $x$ and every $g\in G$ with $g^{-1}\in \dom(x)$ (i.e., $e\in \dom(gx)$). 
\end{enumerate}

\begin{lem}\label{lem:key2}
Let  $\eta$ be an ergodic measurable subgroup of countable group $G$ such that $\eta$-a.e. $H\in 2_e^G$ is infinite. Let $(K,\kappa), (L,\lambda), (M,\mu)$ be standard probability spaces with $H(K,\kappa)=H(L,\lambda)$.  In addition, suppose there exist measurable maps $\alpha:K \to M$ and $\beta:L \to M$ such that $\alpha_*\kappa=\beta_*\lambda=\mu$ and $(M,\mu)$ is nontrivial (i.e., $H(M,\mu)>0$). Then the two Bernoulli shifts $(2_e^G \otimes K, \eta \otimes \kappa)$ and $(2_e^G\otimes L, \eta \otimes \lambda)$ are isomorphic.
\end{lem}

\begin{proof}
The proof is similar to the proof of Lemma \ref{lem:key}. Let $\talpha: 2_e^G \otimes K \to 2_e^G \otimes M$ be the map $\talpha(x)(g):=\alpha(x(g))$ for $g\in \dom(x)$. Define $\tbeta:2_e^G \otimes L \to 2_e^G \otimes M$ similarly. Observe that $\talpha_*(\eta\otimes\kappa)= \tbeta_*(\eta\otimes\lambda)=\eta\otimes\mu$. 


By Theorem \ref{thm:T} there exists an ergodic $U \in \Aut(2_e^G \otimes M,\eta\otimes\mu)$ such that (i) for a.e. $x\in 2_e^G \otimes M$, the orbit $\{U^ix:~ i \in \Z\}$ is infinite, and (ii) for a.e. $x \in 2_e^G \otimes M$, there exists $g\in G$ such that $Ux=gx$.  

Define $T:2_e^G \otimes K \to 2_e^G \otimes K$ by $T(x)=gx$ where $g\in G$ is such that $U(\talpha(x))=g\talpha(x)$. Similarly, let $S:2_e^G \otimes L \to 2_e^G \otimes L$ be defined by $S(y)=gy$ where $g\in G$ is such that $U(\tbeta(y))=g\tbeta(y)$. Note that for a.e. $(x,y)\in 2_e^G \otimes K \times 2_e^G \otimes L$ with $\talpha(x)=\tbeta(y)$, there is a $g\in G$ such that $Tx = gx$ and $Sy=gy$.

Let $\chi_M: 2_e^G \otimes M \to M$ be the projection map $\chi_M(x)=x(e)$. Let $\cB_M$ be the smallest $U$-invariant sigma-algebra on $2_e^G \otimes M$ for which $\chi_M$ is measurable. Define $\chi_K:K^G \to K, \chi_L: L^G \to L, \cB_K, \cB_L$ similarly. 

Observe that $(T,2_e^G\otimes K, \cB_K, \eta \otimes \kappa)$ and $(S,2_e^G\otimes L,\cB_L, \eta \otimes \lambda)$ are Bernoulli relative to $(U,2_e^G\otimes M,\cB_M, \eta \otimes \mu)$. To see this, note, for example, that $\chi_K$ is a generator for $(T,2_e^G\otimes K, \cB_K, \eta \otimes \kappa)$ and $\{\chi_K \circ T^i : i \in \Z\}$ are jointly independent relative to $(\talpha)^{-1}(\cB_M)$. So 
$$h(T| (\talpha)^{-1}(\cB_M) ) = H(\chi_K |~ (\talpha)^{-1}(\cB_M) ) = H(K,\kappa) - H(M,\mu).$$
Similarly, $H(L,\lambda) - H(M,\mu) = h(S| (\tbeta)^{-1}(\cB_M) )$. By Theorem \ref{thm:Thouvenot}, there is an isomorphism $\phi: (2_e^G\otimes K, \cB_K, \eta \otimes \kappa) \to (2_e^G\otimes L,\cB_L, \eta \otimes \lambda)$ such that $\phi T = S \phi$ and $\talpha = \tbeta \phi$. 

Define $\Phi:2_e^G\otimes K \to 2_e^G\otimes L$ by $\Phi(x)(e)= \phi(g^{-1}x)(e)$ for $g\in G$, $x\in 2_e^G\otimes K$. Observe that $\Phi$ is measurable with respect to the Borel sigma-algebras of $2_e^G\otimes K$ and $2_e^G\otimes L$ (which are, in general, larger than $\cB_K$ and $\cB_L$). We claim this is the required isomorphism. It is clearly shift-equivariant.

To see that $\Phi$ is invertible, define $\Psi: 2_e^G\otimes L \to 2_e^G\otimes K$ by $\Psi(y)(g) = \phi^{-1}(g^{-1}y)(e)$. Because both $\Phi$ and $\Psi$ are shift-equivariant and $\Phi \Psi(y)(e)=y(e)$, $\Psi\Phi(x)(e)=x(e)$, it follows that $\Psi$ is the inverse of $\Phi$. It is easy to check that $\tbeta \Phi = \talpha$ and $\Phi T = S \Phi$.

To finish, we must show that $\Phi_*\eta \otimes \kappa = \eta\otimes\lambda$. For $z \in 2_e^G\otimes M$, let $X_z=\{x\in 2_e^G\otimes K:~ \talpha(x)=z\}$ and $Y_z=\{y\in 2_e^G\otimes L:~\tbeta(y)=z\}$. Because $\tbeta \Phi = \talpha$, it follows that $\Phi$ maps $X_z$ to $Y_z$. 

For $z\in 2_e^G\otimes M$, let $\zeta_z$ be the fiber measure of $\eta\otimes\kappa$ over $z$. This family of measures is determined (up to measure zero sets) by the property that $\zeta_z$ is supported on $X_z$ and $\eta\otimes\kappa = \int \zeta_z~d\eta\otimes\mu(z)$. Similarly, let $\nu_z$ be the fiber measure of $\eta\otimes\lambda$ over $z$.

Because $\talpha_*(\eta\otimes\kappa) = \tbeta_*(\eta\otimes\lambda) = \eta\otimes\mu$ and $\Phi$ maps $X_z$ to $Y_z$, in order to show that $\Phi_*(\eta\otimes\kappa) = (\eta\otimes\lambda)$, it suffices to show that $\Phi$ restricts to an isomorphism from $(X_z,\zeta_z)$ to $(Y_z,\nu_z)$ for a.e. $z$.

For $g\in \dom(z)$ and $n\in \Z$ let $\tau_z^n(g) \in G$ be the element satisfying $U^n(g^{-1}z)(e)=z(\tau_z^n(g))$. This is well-defined for a.e. $z\in 2_e^G \otimes M$. If $x\in X_z$ then $T^n(g^{-1}x)(e)=x(\tau_z^n(g))$ by definition.

Fix $z \in 2_e^G \otimes M$. Let $e=g_0,g_1,\ldots \in \dom(z)$ be such that if $O_i = \{\tau^n_z(g_i):~ n \in \Z\}$ then $O_i \cap O_j =\emptyset $ whenever $i\ne j$ and $\cup_{i=0}^\infty O_i=\dom(z)$. So $ K^{\dom(z)} = \prod_{i=0}^\infty K^{O_i}$.

Let $\cB_{K,z,0}$ be the restriction of $\cB_K$ to $X_z$. This is the $\sigma$-algebra of $X_z$ generated by $\{\chi_K \circ T^j :~ j \in \Z\}$. More generally, let $\cB_{K,z,i}$ be the sigma-algebra on $X_z$ generated by $\{\chi_K \circ T^j g_i^{-1}:~ j \in \Z\}$. Define $\sigma$-algebras $\cB_{L,z,i}$ on $Y_z$ similarly. 

Because the $O_i$'s are pairwise disjoint, the sigma-algebras $\cB_{K,z,i}$ are independent. Because $\cup O_i = \dom(z)$, these sigma-algebras generate the Borel sigma-algebra of $K^G$ restricted to $X_z$. Similarly statements apply to $L$ in place of $K$. Therefore, it suffices to show that $\Phi$ restricted to $X_z$ determines an isomorphism from $(X_z,\cB_{K,z,i}, \zeta_z)$ to $(Y_z, \cB_{L,z,i}, \nu_z)$ for all $i$. 

By definition, $\phi$ (and therefore $\Phi$) restricted to $X_z$ is an isomorphism from $(X_z, \cB_{K,z,0}, \zeta_z)$ to $(Y_z, \cB_{L,z,0}, \nu_z)$. Note that $(X_{g_i^{-1}z}, \cB_{K,g^{-1}_iz,0},\zeta_{g_i^{-1}z}) = g^{-1}_i (X_z, \cB_{K,z,i}, \zeta_z)$. Because $\Phi$ is shift-equivariant, this implies  $\Phi$ restricted to $X_z$ determines an isomorphism from $(X_z,\cB_{K,z,i}, \zeta_z)$ to $(Y_z, \cB_{L,z,i}, \nu_z)$ for every $i$.

\end{proof}



\section{Ornstein groups}

If $G$ is a group, $C<G$ a subgroup and $(K,\kappa)$ a probability space then let $(K^{G/C},\kappa^{G/C})$ be the product space with the $G$-action $gx(fC)=x(g^{-1}fC)$. This is called the {\em generalized Bernoulli shift} over $G/C$ with base space $(K,\kappa)$.

\begin{lem}\label{lem:free}
Let $G$ be a countably infinite group and $C<G$ a finite cyclic subgroup of prime order. Let $N(C)=\{g \in G:~gCg^{-1}=C\}$ be the normalizer of $C$. Suppose that $(K,\kappa)$ is a nontrivial probability space. If $G/N(C)$ is infinite then the action $G \cc ( K^{G/C}, \kappa^{G/C})$ is essentially free.
\end{lem}

\begin{proof}
For $g\in G$, let $X_g=\{x\in K^{G/C}:~ gx=x\}$.  It suffices to show that $\kappa^{G/C}(X_g)=0$ for every $g\in G\setminus \{e\}$. 

Given $g\in G$, let $I_g = \{ fC \in G/C:~ gfC \ne fC\}$. It is easy to see that if $|I_g| = \infty$ then there exists a subset $I'_g \subset I_g$ with $|I'_g| = \infty$ such that for every $fC \in I'_g$, $gfC \notin I'_g$. Let $X(g,fC)=\{x \in X:~ x(gfC)=x(fC)\}$. Then the events $\{X(g,fC):~fC \in I'_g\}$ are jointly independent. So
$$\kappa^{G/C}(X_g) \le \kappa^{G/C}\left( \bigcap_{fC \in I'_g} X(g,fC) \right) = \prod_{fC \in I'_g} \kappa^{G/C}(X(g,fC)) = 0.$$
So it suffices to show that $|I_g|=\infty$ for every $g \in G \setminus \{e\}$. (This fact was observed earlier in \cite{KT08} Proposition 2.4).

If $fC \notin I_g$ then $gfC = fC$, i.e., $f^{-1}gf \in C$ which is equivalent to $g \in fCf^{-1}$. In particular, if $g \notin fCf^{-1}$ for any $f$, then $|I_g|=\infty$. So suppose that $g \in fCf^{-1}$ for some $f\in G$ and $g\ne e$. We claim that if $f_2C \notin I_g$ then $f_2N(C) = fN(C)$. Indeed, in this case  $g \in fCf^{-1} \cap f_2Cf_2^{-1}$. Because $C$ is a cyclic group of prime order and $g$ is nontrivial, this implies $fCf^{-1}=f_2Cf_2^{-1}$ which implies $f_2^{-1}f \in N(C)$, i.e., $fN(C)=f_2N(C)$. Because $G/N(C)$ is infinite, $I_g$ is infinite as required.

\end{proof}

If $p_1,\ldots,p_n$ are non-negative real numbers then let $H(p_1,\ldots,p_n):=-\sum_{i=1}^n p_i \log(p_i)$ where $0\log(0):=0$ by convention.

\begin{lem}
For any real numbers $t,r$ with $0<t<1$ and $H(t,1-t)< r$, there exists a standard probability space $(L,\lambda)$ with an element $l_1 \in L$ such that $\lambda(\{l_1\}) = t$ and $H(L,\lambda)=r$.
\end{lem}
\begin{proof}
Let $(N,\nu)$ be a probability space with $H(t,1-t) + (1-t)H(N,\nu) = r$. Let $L$ be the disjoint union of $\{l_1\}$ and $N$. Define the measure $\lambda$ on $L$ by $\lambda(\{l_1\})=t$ and $\lambda(B) = (1-t)\nu(B)$ for all Borel $B\subset N$. Then  $H(L,\lambda) = H(t,1-t) + (1-t)H(N,\nu) = r.$
\end{proof}



\begin{proof}[Proof of Theorem \ref{thm:2atom}]
Let $G$ be a countably infinite group. Suppose $G$ contains a nontrivial element $g_0\in G$ of finite prime order so that if $C=\langle g_0\rangle$ then $G/N(C)$ is infinite. Let $(K,\kappa)$ be a non-trivial two-atom space (e.g., $H(K,\kappa)>0$). By Theorem \ref{thm:main}, to prove $G$ is Ornstein it suffices to prove that there is a non-two-atom space $(L,\lambda)$ with $H(K,\kappa)=H(L,\lambda)>0$ such that the Bernoulli shifts $G \cc (K^G,\kappa^G)$ and $G \cc (L^G,\lambda^G)$ are isomorphic.

Let $p>1$ be the order of $g$. We claim that there is a non-two-atom space $(L,\lambda)$ with $H(L,\lambda)=H(K,\kappa)$ and a standard nontrivial probability space $(M,\mu)$ and factor maps $\alpha:(K^p, \kappa^p) \to (M,\mu)$, $\beta:(L^p,\lambda^p) \to (M,\mu)$. Moreover, we require that if $\sigma_K:K^p \to K^p$ denotes the shift map $\sigma_K(x_0,\ldots,x_{p-1})=(x_1,\ldots,x_{p-1},x_0)$ then $\alpha \sigma_K=\alpha$. We also require that if $\sigma_L:L^p \to L^p$ is defined similarly then $\beta \sigma_L=\beta$.

To prove the claim, we may assume without loss of generality that $K=\{k_0,k_1\}$ with $\kappa(\{k_0\}) \le \kappa(\{k_1\})$. Let $\epsilon=\kappa(\{k_0\})$. By the previous lemma,  there exists a probability space $(L,\lambda)$ with an element $l_1 \in L$ such that $\lambda(\{l_1\}) = (1-\epsilon^p)^{1/p}$ and $H(L,\lambda)=H(K,\kappa)$.  The space $(L,\lambda)$ must not be a two-atom space  since otherwise $H(K,\kappa)=H(\epsilon,1-\epsilon)=H( (1-\epsilon^p)^{1/p}, 1 - (1-\epsilon^p)^{1/p}) = H(L,\lambda)$ implies $1-\epsilon=(1-\epsilon^p)^{1/p}$ which contradicts $0<\epsilon<1$ and $p>1$. Let $M:=\{0,1\}$, $\mu(\{0\}):=\epsilon^p$, $\mu(\{1\}):=1-\epsilon^p$, $\alpha(x_0,\ldots,x_{p-1}) := 0$ if and only if $(x_0,\ldots,x_{p-1}) = (k_0,\ldots,k_0)$ and $\beta(x_0,\ldots,x_{p-1}) :=1$ if and only if $(x_0,\ldots,x_{p-1}) = (l_1,\ldots,l_1)$. This proves the claim.

Let $[K^p] = K^p/\sigma_K$. That is, $[K^p]$ is the set of all $\sigma_K$-orbits in $K^p$. Let $[\kappa^p]$ be the probability measure on $[K^p]$ obtained by pushing $\kappa^p$ forward under the quotient map $K^p \to [K^p]$. Define $([L^p], [\lambda^p])$ similarly. 

 Let $([K^p]^{G/C}, [\kappa^p]^{G/C})$ be the generalized Bernoulli shift over $G/C$ with base space $([K^p], [\kappa^p])$. This system is a factor of $G \cc (K^G,\kappa^G)$ via the map $\pi_K: K^G \to [K^p]^{G/C}$ defined by $\pi_K(x)(gC)=[ x(g),x(gg_0),\ldots, x(gg_0^{p-1})]$ which denotes the $\sigma_K$-orbit of $(x(g),x(gg_0),\ldots, x(gg_0^{p-1}))$. Similarly, define $\pi_L:(L^G,\lambda^G) \to ([L^p]^{G/C},[\lambda^p]^{G/C})$.

Let $\talpha:([K^p]^{G/C}, [\kappa^p]^{G/C}) \to (M^{G/C},\mu^{G/C})$ be the factor map $\talpha(x)(gC)=\alpha(x(gC))$. This involves a slight abuse of notation because $\alpha$ was defined from $K^p$ to $M$ instead of $[K^p]$ to $M$. However because $\alpha \sigma_K=\alpha$, $\alpha$ factors through $[K^p]$. Similarly, define $\tbeta:([L^p]^{G/C}, [\lambda^p]^{G/C}) \to (M^{G/C},\mu^{G/C})$.

Let $\Upsilon: M^{G/C} \to [0,1]$ be a Borel isomorphism. Let $Z$ be the set of all $z\in M^{G/C}$ such that $\Upsilon(z) \le \Upsilon(g^i_0 z)$ for all $i$. Because $\Upsilon$ is Borel, $Z$ is a Borel set. Because the action of $G$ on $(M^{G/C},\mu^{G/C})$ is essentially free (by Lemma \ref{lem:free}), $\{Z, g_0Z, \ldots, g_0^{p-1}Z\}$ partitions $M^{G/C}$ up to a set of measure zero.

If $(A,\rho)$ is a probability space and $B \subset A$ is Borel, then the {\em normalized restriction} of $\rho$ to $B$ is the probability measure $\rho_B$ on $B$ defined by
$$\rho_B(E)=\frac{\rho(E)}{\rho(B)}$$
for all Borel $E \subset B$.

Define $\omega: M^{G/C} \to 2^G$ by $\omega(z) := \{g \in G:~ g^{-1}z \in Z\}$. This map is equivariant: $\omega(fz)=f\omega(z)$. Therefore, $\omega_*\mu^{G/C}$ is an invariant measure on $2^G$ and its normalized restriction $\eta$ to $2_e^G$ is a measurable subgroup. Note that $\eta=\omega_*\mu^{G/C}_Z$ where $\mu^{G/C}_Z$ is the normalized restriction of $\mu^{G/C}$ to $Z$.

The measure $\eta$ is supported on the collection $\cC \subset 2_e^G$ of all subsets $H \subset G$ with the property that $e\in H$ and $\{H, Hg_0, \ldots, Hg_0^{p-1}\}$ is a partition of $G$. In particular, $\eta$-a.e. $H \in 2_e^G$ is infinite.

Let $\cC \otimes K^p \subset 2_e^G \otimes K^p$ be the set of all functions $x:\dom(x) \to K^p$ where $\dom(x) \in \cC$. Define $\cC \otimes L^p$ similarly. Since $\eta \otimes \kappa^p$ is supported on $\cC \otimes K^p$, without loss of generality we may consider it as a measure on $\cC \otimes K^p$. Similarly, $\cC \otimes \lambda^p$ is a measure on $\cC \otimes L^p$.


Let $X=(\talpha \pi_K )^{-1}(Z)$ and $Y=(\tbeta\pi_L)^{-1}(Z)$. Define $\Omega_K:X \to \cC \otimes K^p$ by $\Omega_K(x)(g)=(x(g),x(gg_0),\ldots, x(gg_0^{p-1}))$ for $g\in \dom(\Omega_K(x)):=\omega \talpha \pi_K(x) = \{g\in G:~g^{-1}x\in X\}$. By definition of $\cC$,  $(\Omega_K)_*(\kappa^G_X)=\eta\otimes \kappa^p$ where $\kappa^G_X$ is the normalized restriction of  $\kappa^G$ to $X$. Define $\Omega_L: Y \to \cC \times L^p$ similarly.

For $x\in \cC \otimes K^p$ and $0\le i \le p-1$, let $x_i:\dom(x) \to K$ be the projection to the $i$-th coordinate. Thus $x(g)=(x_0(g),\ldots, x_{p-1}(g))$ for $g\in \dom(x)$. Of course, we define $y_i$ similarly if $y\in \cC \otimes L^p$.

By abuse of notation, we let $\Omega_K^{-1}:\cC \otimes K^p \to K^G$ denote the map $\Omega_K^{-1}(x)(g) = x_i(g')$ where $g',i$ are uniquely determined by: $g=g'g_0^i$, $0\le i \le p-1$, $g' \in \dom(x)$. Note that $\Omega_K^{-1}\Omega_K: X \to X$ is the identity map (but $\Omega_K\Omega_K^{-1}$ is not necessarily well-defined on the domain of $\Omega_K^{-1}$, so $\Omega_K^{-1}$ is only a right-inverse). Define $\Omega_L^{-1}$ similarly.
 
 By Lemma \ref{lem:key2} there is an isomorphism $\phi:(\cC \otimes K^p,\eta\otimes \kappa^p) \to  (\cC \otimes L^p, \eta\otimes \lambda^p)$. 


Define $\Phi:K^G \to L^G$ as follows: if $x\in X$ then $\Phi(x)=\Omega_L^{-1}\phi\Omega_K(x)$. For $0\le i \le p-1$, we define $\Phi(g_0^ix) = g_0^i\Phi(x)$. Because $X,g_0X,\ldots, g_0^{p-1}X$ partitions $K^G$ (up to a set of measure zero), this defines $\Phi$ on a full measure subset of $K^G$ (which is all that we require). We claim that this is the desired isomorphism.

First we show that $\Phi$ is equivariant. If $x\in X, g \in G$ and $gx \in X$ then $\Phi(gx) = g\Phi(x)$ because $\Omega_K, \phi$ and $\Omega_L^{-1}$ are all equivariant (on their domains). So for any $i,j$, 
$$ \Phi((g_0^jgg_0^{-i})g_0^ix) = \Phi(g_0^jgx) = g_0^jg\Phi(x) = (g_0^jgg_0^{-i})\Phi(g_0^ix).$$
This implies $\Phi$ is equivariant because a.e. element $y \in K^G$ can be written as $y=g_0^ix$ (for a unique $0\le i \le p-1$ and $x\in X$) and an arbitrary element $f\in G$ can be written as $f=g_0^jgg_0^{-i}$ for some $g\in G$ with $gx \in X$. Indeed, we let $j$ be determined by the property that $g_0^{-j}fg_0^ix \in X$ then define $g=g_0^{-j}fg_0^i$.

Next we show that $\Phi$ is invertible. For this define $\Psi:L^G \to K^G$ as follows: if $y\in Y$ then $\Psi(x)=\Omega_K^{-1}\phi^{-1}\Omega_L(x)$. For $0\le i \le p-1$, we define $\Psi(g_0^iy) = g_0^i\Psi(y)$. Because $Y,g_0Y,\ldots, g_0^{p-1}Y$ partitions $L^G$ (up to a set of measure zero), this defines $\Psi$ on a full measure subset of $L^G$ (which is all that we require). By an argument similar to the one above, $\Psi$ is equivariant.

If $x\in X$ then $\Psi \Phi x = \Omega_K^{-1}\phi^{-1}\Omega_L \Omega_L^{-1}\phi\Omega_K(x) = x$, i.e., $\Psi\Phi$ restricts to the identity map on $X$. Because $\Phi$ and $\Psi$ are equivariant, $\Psi\Phi$ restricts to the identity map on $g_0^iX$ for every $0\le i \le p-1$. Because $X,g_0X,\ldots, g_0^{p-1}X$ partitions $K^G$ (up to a set of measure zero), this implies $\Psi\Phi$ is the identity map on $K^G$. Similarly, $\Phi\Psi$ is the identity map on $L^G$.

Finally, we claim that $\Phi_*\kappa^G =\lambda^G$. Because $\Phi$ is equivariant it suffices to prove that $\Phi$ restricted to $X$ pushes $\kappa^G_X$ forward to $\lambda^G_Y$. This is true because $(\Omega_K)_*\kappa^G_X = \eta \otimes \kappa^p$, $\phi_*(\eta\otimes \kappa^p) = \eta\otimes \lambda^p$ and $(\Omega_L^{-1})_*(\eta\otimes \lambda^p) = \lambda^G_Y$.
\end{proof}



\begin{thebibliography}{10000000}


\bibitem[Ba05]{Ba05} K. Ball. \textit{Factors of independent and identically distributed processes with non-amenable group actions}.  Ergodic Theory Dynam. Systems  25  (2005),  no. 3, 711--730.




\bibitem[Bo10]{Bo10} L. Bowen. \textit{Measure conjugacy invariants for actions of countable sofic groups}.  J. Amer. Math. Soc. 23 (2010), 217--245. 

\bibitem[Bo11]{Bo11} L. Bowen. \textit{Weak isomorphisms between Bernoulli shifts}.  To appear in Israel Journal of Mathematics. 




\bibitem[DP02]{DP02} A. I. Danilenko and K. K. Park. \textit{Generators and Bernoullian factors for amenable actions and cocycles on their orbits}. Ergodic Theory Dynam. Systems 22 (2002), no. 6, 1715--1745.






\bibitem[Ke10]{Ke10} A. S. Kechris. \textit{Global aspects of ergodic group actions}. Mathematical Surveys and Monographs, 160. American Mathematical Society, Providence, RI, 2010. xii+237 pp.


\bibitem[KL1]{KL1} D. Kerr and H. Li. \textit{Entropy and the variational principle for actions of sofic groups}. arXiv:1005.0399 

\bibitem[KL2]{KL2} D. Kerr and H. Li. \textit{Bernoulli actions and infinite entropy}. arXiv:1005.5143 

\bibitem[Ko58]{Ko58} A. N. Kolmogorov. \textit{ A new metric invariant of transient dynamical systems and automorphisms in Lebesgue spaces}. Dokl. Akad. Nauk SSSR (N.S.)  119  1958 861--864.

\bibitem[Ko59]{Ko59} A. N. Kolmogorov. \textit{ Entropy per unit time as a metric invariant of automorphisms}.  Dokl. Akad. Nauk SSSR  124  1959 754--755. 

\bibitem[KS79]{KS79} M. Keane and M. Smorodinsky. \textit{Bernoulli schemes of the same entropy are finitarily isomorphic}.  Ann. of Math. (2)  109  (1979), no. 2, 397--406.

\bibitem[KT08]{KT08} A. Kechris and T. Tsankov. \textit{Amenable actions and almost invariant sets}. Proc. Amer. Math. Soc. 136(2) (2008), 687--697.



\bibitem[Me93]{Me93} R. Mercer. \textit{The full group of a countable measurable equivalence relation}. 
Proc. Amer. Math. Soc. 117 (1993), no. 2, 323--333.

\bibitem[Ol91]{Ol91} A. Yu. Ol'shanskii. \textit{Geometry of defining relations in groups}. Translated from the 1989 Russian original by Yu. A. Bakhturin. Mathematics and its Applications (Soviet Series), 70. Kluwer Academic Publishers Group, Dordrecht, 1991. xxvi+505 pp. 

\bibitem[Or70a]{Or70a} D. Ornstein. \textit{Bernoulli shifts with the same entropy are isomorphic}.  Advances in Math.  4  (1970) 337--352.

\bibitem[Or70b]{Or70b} D. Ornstein. \textit{Two Bernoulli shifts with infinite entropy are isomorphic}.  Advances in Math.  5  (1970) 339--348.

\bibitem[Os11]{Os11} D. Osin, private communication.

\bibitem[OW87]{OW87} D. Ornstein and B. Weiss. \textit{Entropy and isomorphism theorems for actions of amenable groups}.  J. Analyse Math.  48  (1987), 1--141.


\bibitem[RW00]{RW00} D. J. Rudolph and B. Weiss. \textit{Entropy and mixing for amenable group actions}. Ann. of Math. (2) 151 (2000), no. 3, 1119--1150.

\bibitem[Si59]{Si59} Ya. G. Sina\u\i. \textit{On the concept of entropy for a dynamic system}.   Dokl. Akad. Nauk SSSR  124  (1959) 768--771.

\bibitem[Si62]{Si62} Ya. G. Sina\u\i. \textit{A weak isomorphism of transformations with invariant measure}.  Dokl. Akad. Nauk SSSR  147  (1962) 797--800. 

\bibitem[St75]{St75} A. M. Stepin. \textit{ Bernoulli shifts on groups}.  Dokl. Akad. Nauk SSSR  223  (1975), no. 2, 300--302.

\bibitem[St76]{St76} A. M. Stepin. \textit{ Bernoulli shifts on groups and decreasing sequences of partitions}. Proceedings of the Third Japan-USSR Symposium on Probability Theory (Tashkent, 1975), pp. 592--603. Lecture Notes in Math., Vol. 550, Springer, Berlin, 1976.

\bibitem[Th75]{Th75} J-P. Thouvenot. \textit{Quelques propri\'et\'es des syst\`emes dynamiques qui se d\'ecomposent en un produit de deux syst\`emes dont l'un est un sch\'ema de Bernoulli}. Conference on Ergodic Theory and Topological Dynamics (Kibbutz, Lavi, 1974). Israel J. Math. 21 (1975), no. 2-3, 177--207.






\end{thebibliography}
\end{document}